\documentclass[11pt,reqno]{amsart}

\usepackage{amssymb}
\usepackage{amsmath,amsthm,amsfonts,amssymb,latexsym,mathrsfs,color,hyperref}
\usepackage{graphicx}
\usepackage{xspace}
\usepackage{multirow}
\usepackage{diagbox}
\usepackage{capt-of}
\usepackage{tikz}
\tikzset{
	level 1/.style = {sibling distance = 1.5cm},
	level 2/.style = {sibling distance = 0.8cm},
    level distance = 0.9 cm
}
\usetikzlibrary {decorations.pathmorphing}
\usetikzlibrary{matrix, positioning}
\tikzstyle{snakeline} = [decorate, decoration={snake, amplitude=.4mm, segment length=2mm}]

\usepackage{tikz-cd}
\usepackage[boxsize=1.3em, centerboxes]{ytableau}

\usepackage{tikz-qtree}
\tikzset{every tree node/.style={minimum width=0.1cm,draw,circle},
         blank/.style={draw=none},
         edge from parent/.style=
         {draw,edge from parent path={(\tikzparentnode) -- (\tikzchildnode)}},
         level distance=0.8cm}

\newtheorem{theorem}{Theorem}
\newtheorem{corollary}[theorem]{Corollary}
\newtheorem{proposition}[theorem]{Proposition}

\newtheorem{definition}[theorem]{Definition}

\newcommand{\dd}{{\rm dd}}
\newcommand{\val}{{\rm val}}

\newcommand{\lends}{{\rm exl}}
\newcommand{\mends}{{\rm exm}}
\newcommand{\rends}{{\rm exr}}

\newcommand{\udrun}{{\rm udrun\,}}

\newcommand{\fap}{{\rm fap\,}}

\newcommand{\plat}{{\rm plat\,}}
\newcommand{\ap}{{\rm ap\,}}

\newcommand{\des}{{\rm des\,}}
\newcommand{\cdes}{{\rm cdes\,}}

\newcommand{\exc}{{\rm exc\,}}

\newcommand{\cyc}{{\rm cyc\,}}

\newcommand{\mtn}{\mathcal{T}}

\newcommand{\msn}{\mathfrak{S}_n}

\newcommand{\ms}{\mathfrak{S}}

\newcommand{\lrf}[1]{\lfloor #1\rfloor}

\newcommand{\mqn}{\mathcal{Q}_n}

\newcommand{\asc}{{\rm asc\,}}

\newcommand{\Eulerian}[2]{\genfrac{<}{>}{0pt}{}{#1}{#2}}
\newcommand{\Stirling}[2]{\genfrac{\{}{\}}{0pt}{}{#1}{#2}}
\newcommand{\stirling}[2]{\genfrac{[}{]}{0pt}{}{#1}{#2}}

\title{Normal ordered grammars}
\author[S.-M.~Ma]{Shi-Mei Ma}
\address{School of Mathematics and Statistics,
        Northeastern University at Qinhuangdao,
         Hebei 066000, P.R. China}
\email{shimeimapapers@163.com (S.-M. Ma)}
\author[T.~Mansour]{Toufik Mansour}
\address{Department of Mathematics, University of Haifa, 3498838 Haifa, Israel}
\email{tmansour@univ.haifa.ac.il(T.~Mansour)}
\author{Jean Yeh}
\address{Department of Mathematics, National Kaohsiung Normal University, Kaohsiung 82444, Taiwan}
\email{chunchenyeh@nknu.edu.tw}
\author[Y.-N. Yeh]{Yeong-Nan Yeh}
\address{College of Mathematics and Physics, Wenzhou University, Wenzhou 325035, P.R. China}
\email{mayeh@math.sinica.edu.tw (Y.-N. Yeh)}
\begin{document}

\maketitle
\begin{abstract}
We introduce the theory of normal ordered grammars, which gives a natural generalization of the
normal ordering problem. To illustrate the main idea, we explore normal ordered
grammars associated with the Eulerian polynomials and the second-order Eulerian polynomials. In particular,
we present a normal ordered grammatical interpretation for the
$(\cdes,\cyc)$ $(p,q)$-Eulerian polynomials, where $\cdes$ and $\cyc$ are
the cycle descent and cycle statistics, respectively. 
The exponential generating function for a family of polynomials, generated by a normal ordered
grammar associated with the second-order Eulerian polynomials, reveals an interesting feature:
its expression involves the generating function for Catalan numbers as its exponent.
In the final part, we discuss some normal ordered grammars
related to the type $B$ Eulerian polynomials.
A normal ordered grammatical interpretation of the up-down run polynomial is also established.
\bigskip

\end{abstract}
\date{\today}
\section{Introduction}
The {\it Weyl algebra} $W$ is the unital algebra generated by two symbols $D$ and $U$ satisfying the commutation relation
$DU-UD=I$,
where $I$ is the identity which we identify with ``1". In other words, $W=\langle D,U |DU-UD=I \rangle$.
An example
of the Weyl algebra is the algebra of differential operators acting on the ring of
polynomials in $x$, generated by $D=\frac{\mathrm{d}}{\mathrm{d}x}$ and $U$ acting as multiplication by $x$.
For any $w\in W$, the normal ordering problem is to find the normal
order coefficients $c_{i,j}$ in the expansion:
$$w=\sum_{i,j}c_{i,j}U^iD^j.$$
The following expansion has been studied as early as 1823 by Scherk~\cite[Appendix~A]{Blasiak10}:
\begin{equation}\label{Stirling-def}
(UD)^n=\sum_{k=0}^n\Stirling{n}{k}U^kD^k,
\end{equation}
where $\Stirling{n}{k}$ is the {\it Stirling number of the second kind}, i.e., the number of partitions of the set $[n]=\{1,2,\ldots,n\}$ into $k$ blocks (non-empty subsets).
According to~\cite[Proposition~A.2]{Blasiak10}, one has
\begin{equation}\label{stirling-def}
(\mathrm{e}^xD)^n=\mathrm{e}^{nx}\sum_{k=0}^n\stirling{n}{k}D^k,
\end{equation}
where $\stirling{n}{k}$ is the (signless) Stirling number of the first
kind, i.e.,
the number of permutations of $[n]$ with $k$ cycles.
Many generalizations and variations of~\eqref{Stirling-def} and~\eqref{stirling-def} occur naturally in quantum physics, combinatorics and algebra.
The reader is referred to Schork~\cite{Schork21} for survey and~\cite{Mansour10,Engbers15,Eu17} for recent progress
on this subject.

A {\it context-free grammar} $G$ over an alphabet
$V$ is defined as a set of substitution rules replacing a letter in $V$ by a formal function over $V$.
As usual, the formal function may be a polynomial or a Laurent polynomial.
The formal derivative $D_G$ with respect to $G$ satisfies the derivation rules:
$D_G(u+v)=D_G(u)+D_G(v),~D_G(uv)=D_G(u)v+uD_G(v)$.
Recently, context-free grammars have been widely used, see~\cite{Chen2301,Chen23,Ji24,Ma23,Ma24} for instances.

In this paper, we always let $D_G$ be the formal derivative associated with the grammar $G$. As an illustration, we
recall a classical result, which may be seen as a dual result of~\eqref{Stirling-def}.
\begin{proposition}[{\cite{Chen93}}]\label{grammar01}
If $G=\{a\rightarrow ab, b\rightarrow b\}$, then $D_{G}^n(a)=a\sum_{k=0}^n\Stirling{n}{k}b^k$.
\end{proposition}

The following simple result suggests that it is natural to consider normal ordering problems associated with grammars.
\begin{proposition}\label{prop-3}
If $G=\{x\rightarrow 1\}$, then one has $\left(xD_G\right)^n=\sum_{k=0}^n\Stirling{n}{k}x^kD_{G}^k$.
\end{proposition}

Assume that $u:=u(x,y),v:=v(x,y)$ and $w:=w(x,y)$ are given functions. For the grammar $G=\left\{x\rightarrow u(x,y),~y\rightarrow v(x,y)\right\}$,
we note that the powers of $w(x,y)D_G$ can be expressed as
$$\left(w(x,y)D_G\right)^n=\sum_{k=0}^n \xi_{n,k}(x,y)w^k(x,y)D_G^k.$$

In Section~\ref{section02}, we consider normal ordered grammars associated with the Eulerian polynomials.
In particular, in Theorem~\ref{Anxyq} we find that
if $G=\{x\rightarrow y, y\rightarrow py\}$, then
$$(xD_{G})^n|_{D_{G}=q}=\sum_{\pi\in\msn}x^{n-\exc(\pi)}y^{\exc(\pi)}p^{\cdes(\pi)}q^{\cyc(\pi)},$$
where $\exc,\cdes$ and $\cyc$ are the excedance, cycle descent and cycle statistics, respectively.
In Section~\ref{section03}, we consider normal ordered grammars associated with the second-order Eulerian polynomials.
If $G=\{x\rightarrow y^2, y\rightarrow y^2\}$, one has
$$(xD_{G})^n=\sum_{k=1}^n\sum_{\ell=k}^nC_{n,k,\ell}x^\ell y^{2n-k-\ell}D_{G}^k.$$
Define
$$\widetilde{C}_n(x,y,z)=\sum_{k=1}^n\sum_{\ell=k}^nC_{n,k,\ell}x^\ell y^{2n-k-\ell}z^k,
~\widetilde{C}(x,x,z;t)=\sum_{n=0}^\infty\widetilde{C}_n(x,x,z)\frac{t^n}{n!}.$$
In Theorem~\ref{thm14}, we give a remarkable explicit formula:
$$\widetilde{C}(x,x,z;t)={\mathrm{e}}^{xzt\cdot \operatorname{Cat}(x^2t/2)},$$
where $\operatorname{Cat}(z)=\frac{1-\sqrt{1-4z}}{2z}$ is the generating function for the Catalan numbers.
In Section~\ref{section04}, we discuss some normal ordered grammars
related to the type $B$ Eulerian polynomials. At the end of this paper, we point out that
if $G'=\{x\rightarrow y, y\rightarrow x\}$, then
$$(xD_{G'})^n|_{D_{G'}=1}=y^nT_n\left(\frac{x}{y}\right),$$
where $T_n(x)$ is the up-down run polynomial over permutations in the symmetric group $\msn$.
\section{Normal ordered grammars associated with Eulerian polynomials}\label{section02}
%
The {\it (type $A$) Eulerian polynomials} $A_n(x)$ can be defined by the differential expression:
\begin{equation*}\label{Anx-poly-def}
\left(x\frac{\mathrm{d}}{\mathrm{d}x}\right)^n\frac{1}{1-x}=\sum_{k=0}^\infty k^nx^k=\frac{A_n(x)}{(1-x)^{n+1}}.
\end{equation*}
They satisfy the recurrence relation
\begin{equation}\label{Eulerian01}
A_{n}(x)=nxA_{n-1}(x)+x(1-x)\frac{\mathrm{d}}{\mathrm{d}x}A_{n-1}(x),~A_0(x)=1.
\end{equation}
Let $\msn$ be the {\it symmetric group} of all permutations of $[n]$. For $\pi=\pi(1)\pi(2)\cdots\pi(n)\in\msn$,
the index $i$ is a {\it descent} (resp.~{\it excedance}) if $\pi(i)>\pi(i+1)$ (resp.~$\pi(i)>i$).
Let $\des(\pi)$ and $\exc(\pi)$ be the numbers of descents and excedances of $\pi$, respectively.
The {\it Eulerian polynomials} can also be defined by $$A_n(x)=\sum_{\pi\in\msn}x^{\des(\pi)+1}=\sum_{\pi\in\msn}x^{\exc(\pi)+1}=\sum_{k=1}^n\Eulerian{n}{k}x^k,$$
where $\Eulerian{n}{k}$ are known as the {\it Eulerian numbers} (see~\cite[A008292]{Sloane}).
It is well known that
\begin{equation}\label{Eulerian02}
\Eulerian{n}{k}=k\Eulerian{n-1}{k}+(n-k+1)\Eulerian{n-1}{k-1}.
\end{equation}

In~\cite{Dumont96}, Dumont obtained the context-free grammar for Eulerian polynomials by using a grammatical labeling of circular permutations.
\begin{proposition}[{\cite[Section~2.1]{Dumont96}}]\label{grammar03}
Let $G=\{a\rightarrow ab, b\rightarrow ab\}$.
Then for $n\geqslant 1$, one has
\begin{equation*}
D_{G}^n(a)=D_{G}^n(b)=b^{n+1}A_n\left(\frac{a}{b}\right).
\end{equation*}
\end{proposition}

Note that Proposition~\ref{grammar03} can be restated as
\begin{equation}\label{xDG701}
(xD_{G'})^n(x)=(xD_{G'})^n(y)=y^{n+1}A_n\left(\frac{x}{y}\right),~{\text{where $G'=\{x\rightarrow y,~y\rightarrow y\}$}};
\end{equation}
\begin{equation}\label{xDG702}
(xyD_{G''})^n(x)=(xyD_{G''})^n(y)=y^{n+1}A_n\left(\frac{x}{y}\right),~{\text{where $G''=\{x\rightarrow 1,~y\rightarrow 1\}$}}.
\end{equation}
In order to investigate the powers of $xD_{G'}$ and $xyD_{G''}$, we need to introduce some definitions.
The {\it degree} of a vertex in a tree is referred to the number of its children. We say that $T$ is
a {\it planted binary (resp.~full binary) increasing plane tree} on $[n]$ if it is a binary (resp.~full binary)
tree with $n$ (resp.~$n+1$) unlabeled leaves and $n$ labeled internal vertices, and satisfying the
following conditions (see Figures~\ref{Fig01} and~\ref{Fig03-xy} for examples, where we give every right leaf a weight $y$, and each of the other leaves a weight $x$):
\begin{itemize}
  \item [$(i)$] Internal vertices are labeled by $1,2,\ldots,n$. The node labelled $1$ is distinguished as the root and it has only one child (resp.~it also has two children);
 \item [$(ii)$] Excluding (resp. Including) the root, each internal node has exactly two ordered children, which are referred to as a left child and a right child;
  \item [$(iii)$] For each $2\leqslant i\leqslant n$, the labels of the internal nodes in the unique
path from the root to the internal node labelled $i$ form an increasing sequence.
\end{itemize}

\tikzset{
  solid node/.style={circle,draw,inner sep=1.2,fill=black},
  hollow node/.style={circle,draw,inner sep=1.2},
  level distance = 0.6 cm,
  level 1/.style = {sibling distance = 0.8cm},
  level 2/.style = {sibling distance = 0.6cm},
  level 3/.style = {sibling distance = 0.6cm},
  every level 0 node/.style={draw,hollow node},
  every level 1 node/.style={draw,solid node},
  every level 2 node/.style={draw,solid node},
  every level 3 node/.style={draw,solid node}
}

\begin{figure}[ht!]
\begin{center}
\hspace*{\stretch{1}}
\begin{tikzpicture}
\Tree [.\node[label={1}]{};
        \edge; [.\node[label=left:{2}]{};
            \edge; [.\node[label=left:{3}]{};
                \edge; [.\node [label=below:{$x$}] {}; ]
                \edge; [.\node [label=below:{$y$}] {}; ]]
            \edge; [.\node [label=below:{$y$}] {}; ]]
        ]
\end{tikzpicture};\hspace*{\stretch{1}}
\begin{tikzpicture}
\Tree [.\node[label={1}]{};
        \edge; [.\node[label=left:{2}]{};
        	\edge; [.\node [label=below:{$x$}] {}; ]
            \edge; [.\node[label=right:{3}]{};
                \edge; [.\node [label=below:{$x$}] {}; ]
                \edge; [.\node [label=below:{$y$}] {}; ]]]
        ]
\end{tikzpicture}\hspace*{\stretch{1}}
\end{center}
\caption{The planted binary
increasing plane trees on $[3]$ encoded by $xy^2{D_{G'}}$ and $x^2yD_{G'}$, respectively .}
\label{Fig01}
\end{figure}
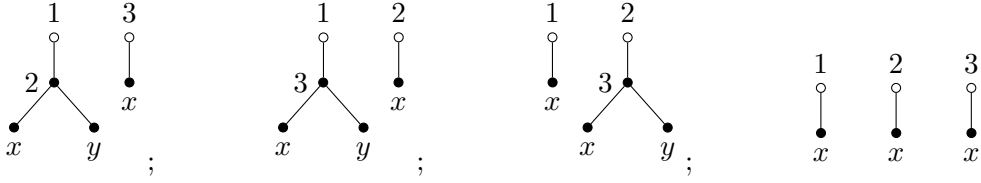
\begin{figure}[ht!]
\begin{center}
\hspace*{\stretch{1}}
\begin{tikzpicture}
    \Tree [.\node[label={1}]{};
        \edge; [.\node[label=left:{2}]{};
        	\edge; [.\node [label=below:{$x$}] {}; ]
        	\edge; [.\node [label=below:{$y$}] {}; ]]
        ]

  \begin{scope}[xshift=1cm]
		\Tree [.\node[label={3}]{};
        	\edge; [.\node[label=below:{$x$}]{};]
        ]
  \end{scope}
\end{tikzpicture};\hspace*{\stretch{1}}
\begin{tikzpicture}
    \Tree [.\node[label={1}]{};
        \edge; [.\node[label=left:{3}]{};
        	\edge; [.\node [label=below:{$x$}] {}; ]
        	\edge; [.\node [label=below:{$y$}] {}; ]]
        ]

  \begin{scope}[xshift=1cm]
    \Tree [.\node[label={2}]{};
        \edge; [.\node[label=below:{$x$}]{};]
        ]
  \end{scope}
\end{tikzpicture};\hspace*{\stretch{1}}
\begin{tikzpicture}
    \Tree [.\node[label={1}]{};
        \edge; [.\node[label=below:{$x$}]{};]
        ]

  \begin{scope}[xshift=1cm]
        \Tree [.\node[label={2}]{};
        	\edge; [.\node[label=left:{3}]{};
        		\edge; [.\node [label=below:{$x$}] {}; ]
        		\edge; [.\node [label=below:{$y$}] {}; ]]
        ]
  \end{scope}
\end{tikzpicture};\hspace*{\stretch{1}}
\begin{tikzpicture}
    \Tree [.\node[label={1}]{};
        \edge; [.\node[label=below:{$x$}]{};]
        ]

  \begin{scope}[xshift=1cm]
        \Tree [.\node[label={2}]{};
        	\edge; [.\node[label=below:{$x$}]{};]
        ]
  \end{scope}

  \begin{scope}[xshift=2cm]
        \Tree [.\node[label={3}]{};
        	\edge; [.\node[label=below:{$x$}]{};]
        ]
  \end{scope}
\end{tikzpicture}\hspace*{\stretch{1}}
\end{center}
\caption{Three 2-forests on $[3]$ encoded by $x^2yD_{G'}^2$, and the 3-forest on $[3]$ encoded by $x^3D_{G'}^3$.}
\label{Fig02}
\end{figure}

\begin{definition}
We say that $F$ is a {\it binary (resp.~full binary) $k$-forest} on $[n]$ if it has $k$ connected components, each connected component is a
planted binary (resp.~full binary) increasing plane tree, the labels of the roots are increasing from left to right and the labels of the $k$-forest form a partition of $[n]$.
\end{definition}
\begin{theorem}\label{Eulerian-tree}
Let $G'=\{x\rightarrow y, y\rightarrow y\}$.
For any $n\geqslant 1$, one has
\begin{equation}\label{aDGthm}
(xD_{G'})^n=\sum_{k=1}^n\sum_{\ell=k}^nA_{n,k,\ell}x^\ell y^{n-\ell}D_{G'}^k,
\end{equation}
where the coefficients $A_{n,k,\ell}$ satisfy the recurrence relation
\begin{equation}\label{Anki-recu}
A_{n+1,k,\ell}=\ell A_{n,k,\ell}+(n-\ell+1)A_{n,k,\ell-1}+A_{n,k-1,\ell-1},
\end{equation}
with the initial conditions $A_{1,1,1}=1$ and $A_{1,k,\ell}=0$ if $(k,\ell)\neq (1,1)$.
The coefficient $A_{n,k,\ell}$ counts binary $k$-forests on $[n]$ with $n-\ell$ right leaves.
\end{theorem}
\begin{proof}
(A) The first few $(xD_{G'})^n$ are given as follows:
\begin{align*}
(xD_{G'})^2&=xyD_{G'}+x^2D_{G'}^2,~(xD_{G'})^3=(xy^2+x^2y)D_{G'}+3x^2yD_{G'}^2+x^3D_{G'}^3,\\
(xD_{G'})^4&=(xy^3+4x^2y^2+x^3y)D_{G'}+(7x^2y^2+4x^3y)D_{G'}^2+6x^3yD_{G'}^3+x^4D_{G'}^4.
\end{align*}
Thus the expansion~\eqref{aDGthm} holds for $n\leqslant 4$. Assume that it holds for $n$.
Since $$(xD_{G'})^{n+1}=xD_{G'}\left(xD_{G'}\right)^n=xD_{G'}\left(\sum_{k=1}^n\sum_{\ell=k}^nA_{n,k,\ell}x^\ell y^{n-\ell}D_{G'}^k\right),$$
it follows that
 \begin{equation}\label{xDGN}
(xD_{G'})^{n+1}=\sum_{k=1}^n\sum_{\ell=k}^nA_{n,k,\ell}\left[\left(\ell x^\ell y^{n-\ell+1}+(n-\ell)x^{\ell+1}y^{n-\ell}\right)D_{G'}^k+x^{\ell+1}y^{n-\ell}D_{G'}^{k+1}\right].
\end{equation}
Extracting the coefficient of $x^\ell y^{n-\ell+1}D_{G'}^k$ on both sides leads to the recursion~\eqref{Anki-recu}.

(B) Let $F$ be a binary $k$-forest.
We first give a labeling of $F$ as follows. Label each planted binary increasing plane tree by $D_{G'}$, a right leaf by $y$, and all the other leaves are labeled by $x$.
The weight of $F$ is defined to be the product of the labels of all trees in $F$. See Figure~\ref{Fig02} for illustrations. Assume that the weight of $F$ is $x^\ell y^{n-\ell}D_{G'}^k$.
Let us examine how to generate a forest $F'$ on $[n+1]$ by adding the vertex $n+1$ to $F$.
We have the following three possibilities:
\begin{itemize}
  \item [$c_1$:] When the vertex $n+1$ is attached to a leaf with label $x$, then $n+1$ becomes a internal node with two children. The weight of $F'$ is $x^{\ell}y^{n-\ell+1}D_{G'}^{k}$;
  \item [$c_2$:] When the vertex $n+1$ is attached to a leaf with label $y$, then $n+1$ becomes a internal node with two children. The weight of $F'$ is $x^{\ell+1}y^{n-\ell}D_{G'}^{k}$;
  \item [$c_3$:] If the vertex $n+1$ is added as a new root, then $F'$ becomes a binary $(k+1)$-forest and the child of $n+1$ has a label $x$.
  The weight of $F'$ is given by $x^{\ell+1}y^{n-\ell}D_{G'}^{k+1}$.
\end{itemize}
As each case corresponds to a term in the right of~\eqref{xDGN}, then $(xD_{G'})^{n+1}$
equals the sum of the weights of all binary $k$-forests on $[n+1]$, where $1\leqslant k\leqslant n+1$. This completes the proof.
\end{proof}

Comparing~\eqref{Anki-recu} with~\eqref{Eulerian02}, we see that $A_{n+1,1,\ell}=\Eulerian{n}{\ell}$.
We define $$A_n(x,y,z)=\sum_{k=1}^n\sum_{\ell=k}^nA_{n,k,\ell}x^\ell y^{n-\ell}z^k.$$
 Multiplying both sides of~\eqref{Anki-recu} by $x^\ell y^{n+1-\ell}z^k$ and summing over all $\ell$ and $k$, we get
\begin{equation}\label{recu-Axyz}
 A_{n+1}(x,y,z)=x(n+z)A_n(x,y,z)+x(y-x)\frac{\partial}{\partial x}A_n(x,y,z),~A_0(x,y,z)=1.
\end{equation}
Combining~\eqref{Eulerian01} and~\eqref{recu-Axyz}, we find that $A_n(x,1,1)=A_n(x)$, where $A_n(x)$ is the Eulerian polynomial. Note that the sum of exponents of $x$ and $y$ equals $n$ in
a general term $x^\ell y^{n-\ell}z^k$.
By induction, it is easy to verify that $yA_n(1,y,1)=A_n(y)$.
Using~\eqref{Anki-recu}, we notice that $A_{n,k,k-1}=0$ and so $A_{n+1,k,k}=kA_{n,k,k}+A_{n,k-1,k-1}$.
Thus $A_{n,k,k}$ satisfies the same
recurrence and initial conditions as $\Stirling{n}{k}$.
In conclusion, we obtain the following result.
\begin{corollary}\label{Dual-result}
For $n\geqslant 1$, we have
\begin{align*}
&\sum_{k=1}^nA_{n,k,k}z^k=\sum_{k=1}^n\Stirling{n}{k}z^k,~yA_{n}(x,y,1)=\sum_{\ell=1}^n\Eulerian{n}{\ell}x^\ell y^{n+1-\ell},\\
&A_n(1,1,z)=z(z+1)\cdots(z+n-1)=\sum_{k=1}^n\stirling{n}{k}z^k,\\
&A_n(x)=A_n(x,1,1)=xA_n(1,x,1)=\frac{\partial}{\partial z}A_{n+1}(x,y,z)|_{y=1,z=0}.
\end{align*}
\end{corollary}

In~\cite{FS70},
Foata and Sch\"utzenberger introduced the $q$-Eulerian polynomials
\begin{equation*}\label{anxq-def}
A_n(x;q)=\sum_{\pi\in\msn}x^{\exc(\pi)}q^{\cyc(\pi)}.
\end{equation*}
The polynomials $A_n(x;q)$ satisfy the recurrence relation (see~\cite[Proposition~7.2]{Brenti00}):
\begin{equation}\label{anxq-rr}
A_{n+1}(x;q)=(nx+q)A_{n}(x;q)+x(1-x)\frac{d}{d x}A_{n}(x;q),~A_{1}(x;q)=1.
\end{equation}

In the following,
we always write permutation by its standard cycle form,
in which each cycle has its smallest
element first and the cycles are written in increasing order of their first elements.
The number of {\it cycle descents} of a permutation is the number of pairs $(a,b)$ where $a$ is the element just before $b$ in its cycle and $a>b$.
Let $\cdes(\pi)$ be the number of cycle descents of $\pi$. For example, $\cdes((1,\mathbf{4},2)(3,5,7)(6,\mathbf{9},8))=2$.
It is clear that
$\exc(\pi)+\cdes(\pi)+\cyc(\pi)=n$ for $\pi\in\msn$.
We can now present a generalization of Theorem~\ref{Eulerian-tree}.
\begin{theorem}\label{Anxyq}
Let $G=\{x\rightarrow y, y\rightarrow py\}$.
For any $n\geqslant 1$, one has
$$(xD_{G})^n|_{D_{G}=q}=\sum_{\pi\in\msn}x^{n-\exc(\pi)}y^{\exc(\pi)}p^{\cdes(\pi)}q^{\cyc(\pi)}.$$
When $p=1$, it reduces to $(xD_{G})^n|_{p=1,D_{G}=q}=A_n(x,y,q)$.
\end{theorem}
\begin{proof}
The first few $(xD_{G})^n$ are listed as follows:
\begin{align*}
(xD_{G})^2&=xyD_{G}+x^2D_{G}^2,~
(xD_{G})^3=(xy^2+px^2y)D_{G}+3x^2yD_{G}^2+x^3D_{G}^3,\\
(xD_{G})^4&=(xy^3+4px^2y^2+p^2x^3y)D_{G}+(7x^2y^2+4px^3y)D_{G}^2+6x^3yD_{G}^3+x^4D_{G}^4.
\end{align*}
Assume the following expansion holds for $n$:
\begin{equation}\label{aDGthm02}
(xD_{G})^n=\sum_{k=1}^n\sum_{\ell=k}^nA_{n,k,\ell}(p)x^\ell y^{n-\ell}D_{G}^k.
\end{equation}
Clearly, $A_{1,1,1}(p)=1$ and $A_{1,k,\ell}(p)=0$ if $(k,\ell)\neq (1,1)$.
Since $$(xD_{G})^{n+1}=xD_{G}\left(xD_{G}\right)^n=xD_{G}\left(\sum_{k=1}^n\sum_{\ell=k}^nA_{n,k,\ell}(p)x^\ell y^{n-\ell}D_{G}^k\right),$$
it follows that
\begin{equation}\label{Anki-recu02}
A_{n+1,k,\ell}(p)=\ell A_{n,k,\ell}(p)+(n-\ell+1)pA_{n,k,\ell-1}(p)+A_{n,k-1,\ell-1}(p).
\end{equation}
which implies that~\eqref{aDGthm02} holds for $n+1$. We claim that
\begin{equation}\label{Ankell}
A_{n,k,\ell}(p)=\sum_{\substack{\pi\in\msn\\ \exc(\pi)=n-\ell\\\cyc(\pi)=k}}p^{\cdes(\pi)}.
\end{equation}
Given a $\pi'\in \ms_{n+1}$.
Suppose $\exc(\pi')=n+1-\ell$ and $\cyc(\pi')=k$.
In order to get $\pi'$ from $\pi\in\ms_{n}$ by inserting the entry $n+1$, there are three ways:
\begin{itemize}
  \item [$(i)$] If $\exc(\pi)=n-\ell$ and $\cyc(\pi)=k$, we can insert $n+1$ right after a drop (i.e., the index $i$ such that $i>\pi(i)$) or a fixed point.
  Note that there are $\ell$ choices for the position of $n+1$. The first term of the right-hand side of~\eqref{Anki-recu02} is explained.
  \item [$(ii)$] If $\exc(\pi)=n+1-\ell$ and $\cyc(\pi)=k$, we can insert $n+1$ right after an excedance. This means we have $n+1-\ell$ choices for the position of $n+1$. Note that the number of cycle descents will increase $1$. The second term in the right hand side of~\eqref{Anki-recu02} is explained.
  \item [$(iii)$] If $\exc(\pi)=n+1-\ell$ and $\cyc(\pi)=k-1$, we can insert $n+1$ right after $\pi$ as a fixed point. The last term in the right hand side is explained.
\end{itemize}
This completes the proof of~\eqref{Ankell}.
\end{proof}

As a variant of Theorem~\ref{Eulerian-tree}, we now present the following result.
\begin{theorem}\label{Eulerian-fulltree}
Let $G''=\{x\rightarrow 1, y\rightarrow 1\}$.
For any $n\geqslant 1$, we have
\begin{equation}\label{full-xy}
(xyD_{G''})^n=\sum_{k=1}^n\sum_{\ell=k}^na_{n,k,\ell}x^\ell y^{n+k-\ell}D_{G''}^k,
\end{equation}
where the coefficients $a_{n,k,\ell}$ satisfy the recurrence relation
\begin{equation}\label{anki-recu}
a_{n+1,k,\ell}=\ell a_{n,k,\ell}+(n+k-\ell+1)a_{n,k,\ell-1}+a_{n,k-1,\ell-1},
\end{equation}
with the initial conditions $a_{1,1,1}=1$ and $a_{1,k,\ell}=0$ if $(k,\ell)\neq (1,1)$.
The coefficient $a_{n,k,\ell}$ counts full binary $k$-forests on $[n]$ with $\ell$ left leaves.
Moreover, we have
\begin{equation}\label{xyG3}
(xyD_{G''})^n=\sum_{k=1}^n\sum_{\ell=k}^{\lrf{(n+k)/2}}\gamma(n,k,\ell)(xy)^\ell (x+y)^{n+k-2\ell}D_{G''}^k,
\end{equation}
where the coefficients $\gamma(n,k,\ell)$ satisfy the recursion
\begin{equation}\label{xyGmma}
\gamma(n+1,k,\ell)=\ell\gamma(n,k,\ell)+2(n+k-2\ell+2)\gamma(n,k,\ell-1)+\gamma(n,k-1,\ell-1),
\end{equation}
with the initial conditions $\gamma(1,1,1)=1$ and $\gamma(1,k,\ell)=0$ for all $(k,\ell)\neq (1,1)$.
\end{theorem}
\begin{proof}
(A) The first few $(xyD_{G''})^n$ are given as follows:
\begin{align*}
(xyD_{G''})^2&=(xy^2+x^2y)D_{G''}+x^2y^2D_{G''}^2,\\
(xyD_{G''})^3&=(xy^3+4x^2y^2+x^3y)D_{G''}+(3x^2y^3+3x^3y^2)D_{G''}^2+x^3y^3D_{G''}^3,\\
(xyD_{G''})^4&=(xy^4+11x^2y^3+11x^3y^2+x^4y)D_{G''}+(7x^2y^4+22x^3y^3+7x^4y^2)D_{G''}^2+\\
&(6x^3y^4+6x^4y^3)D_{G''}^3+x^4y^4D_{G''}^4.
\end{align*}
Thus~\eqref{full-xy} holds for $n\leqslant 4$. Assume that the expansion holds for $n$.
Then we have
\begin{align*}
&(xyD_{G''})^{n+1}\\
&=xyD_{G''}\left(\sum_{k=1}^n\sum_{\ell=k}^na_{n,k,\ell}x^\ell y^{n+k-\ell}D_{G''}^k\right)\\
&=\sum_{k=1}^n\sum_{\ell=k}^na_{n,k,\ell}\left[\left(\ell x^\ell y^{n+k-\ell+1}+(n+k-\ell)x^{\ell+1}y^{n+k-\ell}\right)D_{G''}^k+x^{\ell+1}y^{n+k-\ell+1}D_{G''}^{k+1}\right].
\end{align*}
Extracting the coefficient of $x^\ell y^{n+k-\ell+1}D_{G''}^k$ on both sides leads to the recursion~\eqref{anki-recu}.

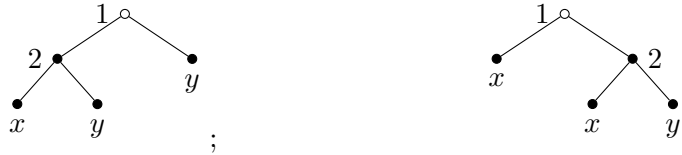
\begin{figure}[ht!]
\begin{center}
\hspace*{\stretch{1}}
\begin{tikzpicture}
\Tree [.\node[label=left:{1}]{};
            \edge; [.\node[label=left:{2}]{};
                \edge; [.\node [label=below:{$x$}] {}; ]
                \edge; [.\node [label=below:{$y$}] {}; ]]
            \edge; [.\node [label=below:{$y$}] {}; ]]
        ]
\end{tikzpicture};\hspace*{\stretch{1}}
\begin{tikzpicture}
\Tree [.\node[label=left:{1}]{};
        	\edge; [.\node [label=below:{$x$}] {}; ]
            \edge; [.\node[label=right:{2}]{};
                \edge; [.\node [label=below:{$x$}] {}; ]
                \edge; [.\node [label=below:{$y$}] {}; ]]]
        ]
\end{tikzpicture}\hspace*{\stretch{1}}
\end{center}
\caption{The planted full binary
increasing plane trees on $[2]$ encoded by $xy^2{D_{G''}}$ and $x^2yD_{G''}$, respectively .}
\label{Fig03-xy}
\end{figure}

(B) Let $F$ be a full binary $k$-forest.
We first give a labeling of $F$ as follows. Label each planted full binary increasing plane tree by $D_{G''}$, a left leaf by $x$ and a right leaf by $y$.
The weight of $F$ is defined to be the product of the labels of all trees in $F$. See Figure~\ref{Fig03-xy} for illustrations. Assume that the weight of $F$ is $x^\ell y^{n+k-\ell}D_{G''}^k$.
Let us examine how to generate a forest $F'$ on $[n+1]$ by adding the vertex $n+1$ to $F$.
We have the following three possibilities:
\begin{itemize}
  \item [$c_1$:] When the vertex $n+1$ is attached to a leaf with label $x$, then $n+1$ becomes a internal node with two children.
  The weight of $F'$ is $x^{\ell}y^{n+k-\ell+1}D_{G''}^{k}$;
  \item [$c_2$:] When the vertex $n+1$ is attached to a leaf with label $y$, then $n+1$ becomes a internal node with two children.
  The weight of $F'$ is $x^{\ell+1}y^{n+k-\ell}D_{G''}^{k}$;
  \item [$c_3$:] If the vertex $n+1$ is added as a new root, then $F'$ becomes a full binary $(k+1)$-forest, the left child of $n+1$ has a label $x$,
  while the right child of $n+1$ has a label $y$.
  The weight of $F'$ is given by $x^{\ell+1}y^{n+k-\ell+1}D_{G''}^{k+1}$.
\end{itemize}
The above three cases exhaust all the possibilities. Thus $(xyD_{G''})^{n+1}$
equals the sum of the weights of all full binary $k$-forests on $[n+1]$, where $1\leqslant k\leqslant n+1$.

(C) We now consider a change of the grammar $G''$. Setting $u=xy$ and $v=x+y$, we get
$$D_{G''}(u)=D_{G''}(xy)=v,~D_{G''}(v)=D_{G''}(x+y)=2.$$
Let $G'''=\{u\rightarrow v,~v\rightarrow 2\}$. Then we have
$\left(xyD_{G''}\right)^n=\left(uD_{G'''}\right)^n$.
Note that
$$\left(uD_{G'''}\right)^2=uvD_{G'''}+u^2D_{G'''}^2,~\left(uD_{G'''}\right)^3=(uv^2+2u^2)D_{G'''}+3u^2vD_{G'''}^2+u^3D_{G'''}^3.$$
By induction, it is easy to check that
\begin{equation*}
\left(uD_{G'''}\right)^n=\sum_{k=1}^n\sum_{\ell=k}^{\lrf{(n+k)/2}}\gamma(n,k,\ell)u^\ell v^{n+k-2\ell}D_{G'''}^k,
\end{equation*}
where the coefficients $\gamma(n,k,\ell)$ satisfy the recursion~\ref{xyGmma}.
Then upon taking $u=xy$ and $v=x+y$, we get~\eqref{xyG3}.
This completes the proof.
\end{proof}

Comparing~\eqref{anki-recu} with~\eqref{Eulerian02}, we notice that $a_{n,1,\ell}=\Eulerian{n}{\ell}$.
Define $$a_n(x,y,z)=\sum_{k=1}^n\sum_{\ell=k}^na_{n,k,\ell}x^\ell y^{n+k-\ell}z^k,~a_0(x,y,z)=1.$$
Multiplying both sides of~\eqref{anki-recu} by $x^\ell y^{n+k-\ell+1}z^k$ and summing over all $\ell$ and $k$, we obtain
$$a_{n+1}(x,y,z)=x(n+yz)a_n(x,y,z)+x(y-x)\frac{\partial}{\partial x}a_n(x,y,z)+xz\frac{\partial}{\partial z}a_n(x,y,z).$$
In particular,
$$a_{n+1}(1,1,z)=(n+z)a_n(1,1,z)+z\frac{\mathrm{d}}{\mathrm{d} z}a_n(1,1,z),~a_0(1,1,z)=1.$$
Let $a_n(1,1,z)=\sum_{k=1}^nL(n,k)z^k$. It follows that
$L(n+1,k)=(n+k)L(n,k)+L(n,k-1)$,
from which we notice that $L(n,k)$ is the (signless) {\it Lah number}, see~\cite{Engbers19} for instance. Explicitly, $$L(n,k)=\binom{n-1}{k-1}\frac{n!}{k!}.$$
\begin{corollary}
For $n\geqslant 1$,
we have $$a_n(1,1,z)=\sum_{k=1}^n\binom{n-1}{k-1}\frac{n!}{k!}z^k.$$
\end{corollary}

A partition of $[n]$ into {\it lists} is a set
partition of $[n]$ for which the elements of each block are linearly ordered.
It is well known that $L(n,k)$ counts set partitions of $[n]$ into $k$ lists (see~\cite[A008297]{Sloane}).
We always assume that each list is prepended and appended by $0$. Given a list $\sigma_1\sigma_2\cdots\sigma_i$.
We identify it with the word $0\sigma_1\sigma_2\cdots\sigma_i0$.
We say that an index $p\in \{0,1,2,\ldots,i-1\}$ is an {\it ascent} if $\sigma_p<\sigma_{p+1}$, and $q\in \{1,2,\ldots,i\}$ is a {\it descent} if $\sigma_p>\sigma_{p+1}$, where we set $\sigma_0=\sigma_{i+1}=0$.
Let $F$ be a full binary $k$-forest. Following~\cite[p.~51]{Stanley11}, a bijection from full binary $k$-forests to set partitions with $k$ lists can be given as follows: Read the internal vertices of trees (from left to right) of $F$ in symmetric order, i.e.,
read the labels of the left subtree (in symmetric order, recursively),
then the label of the root, and then the labels of the right subtree. Using this correspondence, we get the following result.
\begin{corollary}
Let $a_{n,k,\ell}$ be defined by~\eqref{full-xy}. Then $a_{n,k,\ell}$ is the number of set partitions of $[n]$ into $k$ lists with $\ell$ ascents and $n+k-\ell$ descents.
\end{corollary}

For a permutation $\pi\in\msn$ with $\pi(0)=\pi(n+1)=0$, we say that the entry $\pi(i)$
\begin{itemize}
  \item is a {\it valley} if $\pi(i-1)>\pi(i)<\pi(i+1)$;
  \item is a {\it double descent} if $\pi(i-1)>\pi(i)>\pi(i+1)$.
\end{itemize}
Let $\val(\pi)$ (resp.~$\dd(\pi)$) denote the number of valleys (resp.~double descents) in $\pi$.
Define $$\gamma(n,\ell)=\#\{\pi\in\msn: \val(\pi)=\ell,~\dd(\pi)=0\}.$$
A classical result of Foata-Sch\"utzenberger~\cite{Foata73} states that the Eulerian polynomials
have the following $\gamma$-expansion:
$$A_n(x)=x\sum_{\ell=0}^{\lrf{(n-1)/2}}\gamma(n,\ell)x^\ell(1+x)^{n-1-2\ell}.$$
Br\"and\'{e}n~\cite{Branden08} reproved this expansion by introducing the modified Foata-Strehl action.
Let $\mathcal{S}(n,k)$ be the set of partitions of $[n]$ into $k$ lists.
Applying the modified Foata-Strehl action on each list of an element in $\mathcal{S}(n,k)$, we find the following result, and omit the proof for simplicity.
\begin{corollary}
For $n\geqslant 1$, the polynomials $a_n(x,y,z)$ is partial $\gamma$-positive, i.e.,
\begin{align*}
\sum_{k=1}^n\sum_{\ell=k}^na_{n,k,\ell}x^\ell y^{n+k-\ell}z^k&=\sum_{k=1}^nz^k\sum_{\ell=k}^{\lrf{(n+k)/2}}\gamma(n,k,\ell)(xy)^\ell (x+y)^{n+k-2\ell}\\
&=\sum_{k=1}^n(xyz)^k\sum_{i=0}^{\lrf{(n-k)/2}}\gamma(n,k,k+i)(xy)^i (x+y)^{n-k-2i},
\end{align*}
where $\gamma(n,k,k+i)$ counts partitions of $[n]$ into $k$ lists with $i$ valleys and with no double descents.
\end{corollary}
\section{Normal ordered grammars associated with second-order Eulerian polynomials}\label{section03}
Following Carlitz~\cite{Carlitz65}, the {\it second-order Eulerian polynomials} $C_n(x)$ are defined by
$$\sum_{k=0}^\infty \Stirling{n+k}{k}x^k=\frac{C_n(x)}{(1-x)^{2n+1}},$$
which have been well studied in recent years, see~\cite{Carlitz65,Chen17,Dzhumadil14,Elizalde,Haglund12,Ma23}.

For $\mathbf{m}=(m_1,m_2,\ldots,m_n)\in \mathbb{N}^n$, let $\mathbf{n}=\{1^{m_1},2^{m_2},\ldots,n^{m_n}\}$ be a multiset,
where $i$ appears $m_i$ times.
We say that a multipermutation $\sigma$ of $\mathbf{n}$ is {\it Stirling permutation} if $\sigma_s\geqslant\sigma_i$ as soon as $\sigma_i=\sigma_j$ and $i<s<j$.
Denote by $\mqn$ the set of Stirling permutations of $\{1^{2},2^{2},\ldots,n^{2}\}$.
Let $\sigma=\sigma_1\sigma_2\cdots\sigma_{2n}\in\mqn$. In the following discussion, we always set $\sigma_0=\sigma_{2n+1}=0$.
For $0\leqslant i\leqslant 2n$, we say that an index $i$ is a {\it descent} (resp.~{\it ascent}, {\it plateau}) of $\sigma$ if
$\sigma_{i}>\sigma_{i+1}$ (resp. $\sigma_{i}<\sigma_{i+1}$, $\sigma_{i}=\sigma_{i+1}$).
Let $\des(\sigma),\asc(\sigma)$ and $\plat(\sigma)$ be the number of descents, ascents and plateaus of $\sigma$, respectively.
It is now well known that descents, ascents and plateaus have the same distribution over $\mqn$, and their common enumerative polynomials are the second-order Eulerian polynomials $C_n(x)$.
As a variant of~\cite[Theorem~2.3]{Chen17},
the grammatical description of $C_n(x)$ can be restated as follows:
\begin{equation}\label{xyCn}
(xD_{G})^n(x)=y^{2n+1}C_n\left(\frac{x}{y}\right),~\text{where $G=\{x\rightarrow y^2,~y\rightarrow y^2\}$.}
\end{equation}

We say that $T$ is
a {\it planted ternary (resp.~full ternary) increasing plane tree} on $[n]$ if it is a ternary tree with $2n-1$ (resp.~$2n+1$) unlabeled leaves and $n$ labeled internal vertices, and satisfying the following conditions (see~Figures~\ref{Fig01-ternary} and~\ref{Fig05-xyz}, where we give each leaf a weight):
\begin{itemize}
  \item [$(i)$] Internal vertices are labeled by $1,2,\ldots,n$. The node labelled $1$ is distinguished as the root and it has only one child (resp.~it also has three children);
 \item [$(ii)$] Excluding (resp.~Including) the root, each internal node has exactly three ordered children, which are referred to as a left child, a middle child and a right child;
  \item [$(iii)$] For each $2\leqslant i\leqslant n$, the labels of the internal nodes in the unique
path from the root to the internal node labelled $i$ form an increasing sequence.
\end{itemize}
We say that $F$ is a {\it ternary (resp.~full ternary) $k$-forest} on $[n]$ if it has $k$ connected components, each component is a
planted ternary (resp.~full ternary) increasing plane tree, the labels of the roots are increasing from left to right and the labels of the $k$-forest form a partition of $[n]$.

\begin{figure}[ht!]
\begin{center}
\hspace*{\stretch{1}}
\begin{tikzpicture}
\Tree [.\node[label={1}]{};
        \edge; [.\node[label=left:{2}]{};
            \edge; [.\node[label=left:{3}]{};
                \edge; [.\node [label=below:{$x$}] {}; ]
                   \edge; [.\node [label=below:{$y$}] {}; ]
                \edge; [.\node [label=below:{$y$}] {}; ]]
                    \edge; [.\node [label=below:{$y$}] {}; ]
            \edge; [.\node [label=below:{$y$}] {}; ]]
        ]
\end{tikzpicture};\hspace*{\stretch{1}}
\begin{tikzpicture}
\Tree [.\node[label={1}]{};
        \edge; [.\node[label=left:{2}]{};
        	\edge; [.\node [label=below:{$x$}] {}; ]
            \edge; [.\node[label=right:{3}]{};
                \edge; [.\node [label=below:{$x$}] {}; ]
                \edge; [.\node [label=below:{$y$}] {}; ]
                \edge; [.\node [label=below:{$y$}] {}; ]]
                \edge; [.\node [label=below:{$y$}] {}; ]]
        ]
\end{tikzpicture};\hspace*{\stretch{1}}
\begin{tikzpicture}
\Tree [.\node[label={1}]{};
        \edge; [.\node[label=left:{2}]{};
        	\edge; [.\node [label=below:{$x$}] {}; ]
          	\edge; [.\node [label=below:{$y$}] {}; ]
            \edge; [.\node[label=right:{3}]{};
                \edge; [.\node [label=below:{$x$}] {}; ]
                  \edge; [.\node [label=below:{$y$}] {}; ]
                \edge; [.\node [label=below:{$y$}] {}; ]]]
        ]
\end{tikzpicture}\hspace*{\stretch{1}}
\end{center}
\caption{The planted ternary
increasing plane trees on $[3]$ encoded by $xy^4{D_{G}}$ and $x^2y^3D_{G}$, respectively .}
\label{Fig01-ternary}
\end{figure}

Let $F$ be a ternary $k$-forest.
We introduce a labeling of $F$ as follows (see Figure~\ref{Fig01-ternary} for illustrations). Label each planted ternary increasing plane tree by $D_{G}$,
a left leaf by $x$, and middle and right leaves are both labeled by $y$.
If a tree has only one internal vertex and a leaf, then label the leaf by $x$.
Along the same lines as in the proof of Theorem~\ref{Eulerian-tree}, it is routine to verify the following.
\begin{theorem}\label{Eulerian-second-tree}
Let $G=\{x\rightarrow y^2, y\rightarrow y^2\}$.
For any $n\geqslant 1$, we have
\begin{equation*}\label{xDG9thm}
(xD_{G})^n=\sum_{k=1}^n\sum_{\ell=k}^nC_{n,k,\ell}x^\ell y^{2n-k-\ell}D_{G}^k,
\end{equation*}
where the coefficients $C_{n,k,\ell}$ satisfy the recurrence relation
\begin{equation}\label{Cnki-recu}
C_{n+1,k,\ell}=\ell C_{n,k,\ell}+(2n-k-\ell+1)C_{n,k,\ell-1}+C_{n,k-1,\ell-1},
\end{equation}
with the initial conditions $C_{1,1,1}=1$ and $C_{1,k,\ell}=0$ if $(k,\ell)\neq (1,1)$.
The coefficient $C_{n,k,\ell}$ counts ternary $k$-forests on $[n]$ with $2n-k-\ell$ middle and right leaves.
Moreover, we have $C_{n+1,1,\ell}=C_{n,\ell}$, where $C_{n,\ell}$ is the second-order Eulerian number, i.e., the number of Stirling permutations of order $n$ with $\ell$ descents.
\end{theorem}

Define $$\widetilde{C}_n(x,y,z)=\sum_{k=1}^n\sum_{\ell=k}^nC_{n,k,\ell}x^\ell y^{2n-k-\ell}z^k.$$
It follows from~\eqref{Cnki-recu} that
$$\widetilde{C}_{n+1}(x,y,z)=(xz+2nxy)\widetilde{C}_n(x,y,z)+xy(y-x)\frac{\partial}{\partial x}\widetilde{C}_n(x,y,z)-xyz\frac{\partial}{\partial z}\widetilde{C}_n(x,y,z),$$
with $\widetilde{C}_0(x,y,z)=1$.
When $x=y$, one has
\begin{equation}\label{Cn11z02}
\widetilde{C}_{n+1}(x,x,z)=(xz+2nx^2)\widetilde{C}_n(x,x,z)-x^2z\frac{\partial}{\partial z}\widetilde{C}_n(x,x,z),
\end{equation}
Let $$\widetilde{C}(x,x,z;t)=\sum_{n=0}^\infty\widetilde{C}_n(x,x,z)\frac{t^n}{n!}.$$
Then~\eqref{Cn11z02} can be written as
$$(1-2x^2t)\frac{\partial}{\partial t}\widetilde{C}(x,x,z;t)=xz\widetilde{C}(x,x,z;t)-x^2z\frac{\partial}{\partial z}\widetilde{C}(x,x,z;t),~\widetilde{C}(x,x,z;0)=1.$$
With help of mathematical programming, we find the following result.
\begin{theorem}\label{thm14}
We have
$$\widetilde{C}(x,x,z;t)={\mathrm{e}}^{xzt\cdot \operatorname{Cat}(x^2t/2)},$$
where $\operatorname{Cat}(z)=\frac{1-\sqrt{1-4z}}{2z}$ is the generating function for the Catalan numbers $\frac{1}{n+1}\binom{2n}{n}$.
\end{theorem}
\begin{corollary}\label{propT}
For all $n\geqslant 0$, we have
$$\widetilde{C}_{n+1}(x,x,z)=\sum_{j=0}^n\frac{(n+j)!}{2^j(n-j)!j!}x^{n+1+j}z^{n+1-j}=\sum_{j=0}^nb(n,j)x^{n+1+j}z^{n+1-j},$$
where $b(n,j)$ is the {\it Bessel number of first kind}~\cite[A001498]{Sloane}.
\end{corollary}
\begin{proof}
Using~\cite[Eq.~(2.5.16)]{W1}, we get
\begin{align*}
\widetilde{C}(x,x,z;t)&=\sum_{j\geq0}\frac{x^jz^jt^j\operatorname{Cat}^j(x^2t/2)}{j!}\\
&=1+\sum_{j\geq1}\sum_{i\geq0}\frac{j}{(i+j)j!2^i}\binom{2i-1+j}{i}x^{2i+j}z^jt^{i+j}\\
&=1+\sum_{i\geq0}\sum_{j=0}^i\frac{j+1}{(i+1)(j+1)!2^{i-j}}\binom{2i-j}{i-j}x^{2i-j+1}z^{j+1}t^{i+1}.
\end{align*}
Hence, for all $n\geqslant1$, we get $$\widetilde{C}_n(x,x,z)=n!\sum_{j=0}^{n-1}\frac{j+1}{n(j+1)!2^{n-1-j}}\binom{2n-j-2}{n-1-j}x^{2n-j-1}z^{j+1},$$
which is equivalent to
$$\widetilde{C}_n(x,x,z)=\sum_{j=0}^{n-1}\frac{j!}{2^j}\binom{n-1}{j}\binom{n+j-1}{j}x^{n+j}z^{n-j}.$$
After simplifying, we get the desired explicit formula.
\end{proof}

The {\it trivariate second-order Eulerian polynomials} are defined by
\begin{equation*}\label{Cnxyz011}
C_n(x,y,z)=\sum_{\sigma\in\mqn}x^{\asc{(\sigma)}}y^{\des(\sigma)}z^{\plat{(\sigma)}}.
\end{equation*}
In~\cite[p.~317]{Dumont80},
Dumont found that
\begin{equation}\label{Dumont80}
C_{n+1}(x,y,z)=xyz\left(\frac{\partial}{\partial x}+\frac{\partial}{\partial y}+\frac{\partial}{\partial z}\right)C_n(x,y,z),
\end{equation}
which implies that $C_n(x,y,z)$ is symmetric in the variables $x,y$ and $z$.
By~\eqref{Dumont80}, it is clear that
\begin{equation}\label{grammar-Stirling}
D_{G}^n(x)=C_n(x,y,z),
\end{equation}
where $G=\{x \rightarrow xyz, y\rightarrow xyz, z\rightarrow xyz\}$.
In~\cite{Haglund12}, Haglund-Visontai introduced a refinement of the polynomial $C_n(x,y,z)$ by indexing each ascent,
descent and plateau by the values where they appear.
Using~\eqref{grammar-Stirling}, Chen-Fu~\cite{Chen22}
found that $C_n(x,y,z)$ is $e$-positive, i.e.,
\begin{equation}\label{Cnxyz}
C_n(x,y,z)=\sum_{i+2j+3k=2n+1}\gamma_{n,i,j,k}(x+y+z)^{i}(xy+yz+zx)^{j}(xyz)^k,
\end{equation}
where the coefficient $\gamma_{n,i,j,k}$ equals the number of 0-1-2-3 increasing plane trees
on $[n]$ with $k$ leaves, $j$ degree one vertices and $i$ degree two vertices.

A {\it ternary increasing tree} of size $n$ is an increasing plane tree with $3n+1$ nodes where each
interior node is labeled and has three children (a left child,
a middle child and a right child), while exterior nodes have no children and no labels.
Let $\mtn_n$ denote the set of ternary increasing trees of size $n$, see Figure~\ref{Fig05-xyz} for instance. For any $T\in\mtn_n$,
it is clear that $T$ has exactly $2n+1$ exterior nodes.
Let $\lends(T)$ (resp.~$\mends(T)$,~$\rends(T)$) denotes the number of exterior left nodes (resp.~exterior middle nodes, exterior right nodes) in $T$.
Using a recurrence relation that is equivalent to~\eqref{Dumont80}, Dumont~\cite[Proposition~1]{Dumont80} found that
\begin{equation}\label{QnTn}
C_n(x,y,z)=\sum_{T\in\mtn_n}x^{\lends(T)}y^{\mends(T)}z^{\rends(T)}.
\end{equation}

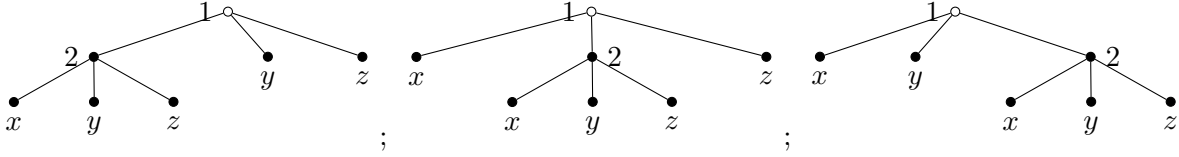
\begin{figure}[ht!]
\begin{center}
\hspace*{\stretch{1}}
\begin{tikzpicture}
\Tree [.\node[label=left:{1}]{};
            \edge; [.\node[label=left:{2}]{};
                \edge; [.\node [label=below:{$x$}] {}; ]
                 \edge; [.\node [label=below:{$y$}] {}; ]
                \edge; [.\node [label=below:{$z$}] {}; ]]
                 \edge; [.\node [label=below:{$y$}] {}; ]
            \edge; [.\node [label=below:{$z$}] {}; ]]
        ]
\end{tikzpicture};\hspace*{\stretch{1}}
\begin{tikzpicture}
\Tree [.\node[label=left:{1}]{};
        	\edge; [.\node [label=below:{$x$}] {}; ]
            \edge; [.\node[label=right:{2}]{};
                \edge; [.\node [label=below:{$x$}] {}; ]
                  \edge; [.\node [label=below:{$y$}] {}; ]
                \edge; [.\node [label=below:{$z$}] {}; ]]
                \edge; [.\node [label=below:{$z$}] {}; ]]
        ]
\end{tikzpicture};\hspace*{\stretch{1}}
\begin{tikzpicture}
\Tree [.\node[label=left:{1}]{};
        	\edge; [.\node [label=below:{$x$}] {}; ]
        \edge; [.\node [label=below:{$y$}] {}; ]
            \edge; [.\node[label=right:{2}]{};
                \edge; [.\node [label=below:{$x$}] {}; ]
                  \edge; [.\node [label=below:{$y$}] {}; ]
                \edge; [.\node [label=below:{$z$}] {}; ]]]
        ]
\end{tikzpicture}\hspace*{\stretch{1}}
\end{center}
\caption{The planted full ternary
increasing plane trees on $[2]$ encoded by $xy^2z^2{D_{G}},~x^2yz^2D_{G}$ and $x^2y^2zD_{G}$.}
\label{Fig05-xyz}
\end{figure}

Let $F$ be a full ternary $k$-forest.
We now give a labeling of $F$ as follows. Label each planted full ternary increasing plane tree by $D_{G}$,
a left leaf by $x$, a middle leaf by $y$ and a right leaf by $z$, see Figure~\ref{Fig05-xyz}.
Along the same lines as in the proof of Theorem~\ref{Eulerian-tree}, we find the following result.
\begin{theorem}\label{Eulerian-fulltree}
If $G=\{x\rightarrow 1, y\rightarrow 1, z\rightarrow 1\}$, then we have
\begin{equation}\label{full-xyz}
(xyzD_{G})^n=\sum_{k=1}^n\sum_{i=0}^{n-k}\sum_{j=0}^{n-k}\eta_{n,i,j,k}x^i y^{j}z^{2n-2k-i-j}(xyz)^kD_{G}^k,
\end{equation}
where the coefficient $\eta_{n,i,j,k}$ counts full ternary $k$-forests on $[n]$ with $i+k$ left leaves, $j+k$ middle leaves and $2n-k-i-j$ right leaves.
\end{theorem}

\begin{definition}
A partition of $\{1^{2},2^{2},\ldots,n^{2}\}$ into {\it Stirling-lists} is a set
partition of $\{1^{2},2^{2},\ldots,n^{2}\}$ for which the elements of each block are Stirling permutations and for all $i\in [n]$,
the two copies of $i$ appear in exactly one block. We always assume that each Stirling-list is prepended and appended by $0$.
\end{definition}

Let $\operatorname{SL}_n$ denote the set of partition of $\{1^{2},2^{2},\ldots,n^{2}\}$ into Stirling-lists, and let $\operatorname{bk}$ be the block statistic.
For example, $\operatorname{SL}_2=\{\{1122\},\{1221\},\{2211\},\{11\}\{22\}\}$, where the last set partition $\{11\}\{22\}$ has two blocks.
Combining~\eqref{QnTn} and Theorem~\ref{Eulerian-fulltree}, we get the following.
\begin{corollary}
If $G=\{x\rightarrow 1, y\rightarrow 1, z\rightarrow 1\}$, then
\begin{equation*}
(xyzD_{G})^n|_{D_G=q}=\sum_{p\in\operatorname{SL}_n}x^{\asc(p)}y^{\plat}z^{\des(p)}q^{\operatorname{bk}(p)}.
\end{equation*}
In particular, the coefficient $q$ in $(xyzD_{G})^n|_{D_G=q}$ is $C_n(x,y,z)$.
\end{corollary}

\begin{theorem}
Let $G=\{x\rightarrow 1, y\rightarrow 1, z\rightarrow 1\}$.
Then we have
\begin{equation}\label{xyzG3}
(xyzD_{G})^n=\sum_{k=1}^n\sum_{2j+3\ell=0}^{2n-2k}\beta_{n,k,j,\ell}(x+y+z)^{2n-2k-2j-3\ell} (xy+xz+yz)^{j}(xyz)^{\ell+k}D_{G}^k,
\end{equation}
where the coefficients $\beta_{n,k,j,\ell}$ satisfy the recursion
\begin{equation}\label{xyzbeta}
\begin{aligned}
\beta_{n+1,k,j,\ell}&=(\ell+k)\beta_{n,k,j-1,\ell}+2(j+1)\beta_{n,k,j+1,\ell-1}+\\
&3(2n-2k-2j-3\ell+3)\beta_{n,k,j,\ell-1}+\beta_{n,k-1,j,\ell},
\end{aligned}
\end{equation}
with the initial conditions $\beta_{1,1,0,0}=1$ and $\beta_{1,k,j,\ell}=0$ for any $(k,j,\ell)\neq (1,0,0)$.
\end{theorem}
\begin{proof}
Consider a change of the grammar $G=\{x\rightarrow 1, y\rightarrow 1, z\rightarrow 1\}$. Setting $$u=x+y+z,~v=xy+xz+yz,~w=xyz,$$
we have $D_{G}(u)=3,~D_{G}(v)=2u$ and $D_{G}(w)=v$.
Let $G'=\{u\rightarrow 3,~v\rightarrow 2u,~w\rightarrow v\}$. Then we have
$\left(xyzD_{G}\right)^n=\left(wD_{G'}\right)^n$.
Note that
$$\left(wD_{G'}\right)^2=wvD_{G'}+w^2D_{G'}^2,~\left(wD_{G'}\right)^3=(v^2+2wu)wD_{G'}+3vw^2D_{G'}^2+w^3D_{G'}^3.$$
By induction, it is routine to check that there exist nonnegative integers $\alpha_{n,k,i,j,\ell}$ such that
\begin{align*}
\left(wD_{G'}\right)^n&=\sum_{k=1}^n\sum_{i+2j+3\ell=2n-2k}\alpha_{n,k,i,j,\ell}u^i v^{j}w^{\ell+k}D_{G'}^k\\
&=\sum_{k=1}^n\sum_{2j+3\ell=0}^{2n-2k}\beta_{n,k,j,\ell}u^{2n-2k-2j-3\ell} v^{j}w^{\ell+k}D_{G'}^k
\end{align*}
where $\beta_{n,k,j,\ell}$ satisfy the recursion~\ref{xyzbeta}.
Then upon taking $u=x+y+z,~v=xy+xz+yz$ and $w=xyz$, we get~\eqref{xyzG3}.
This completes the proof.
\end{proof}

Let
$$\beta_n:=\beta_n(u,v,w,q)=\sum_{k=1}^n\sum_{2j+3\ell=0}^{2n-2k}\beta_{n,k,j,\ell}u^{2n-2k-2j-3\ell} v^{j}w^{\ell+k}q^k.$$
It follows from~\eqref{xyzbeta} that
$$\beta_{n+1}=wq\beta_n+3w\frac{\partial}{\partial u}\beta_n+2uw\frac{\partial}{\partial v}\beta_n+vw\frac{\partial}{\partial w}\beta_n.$$
Below are these polynomials for $n\leqslant 4$:
\begin{align*}
\beta_1&=wq,~\beta_2=vwq+w^2q^2,~\beta_3=(v^2w+2uw^2)q+ 3vw^2q^2+w^3q^3,\\
\beta_4&=(v^3w+8uvw^2+6w^3)q+(7v^2w^2+8uw^3)q^2+6v w^3q^3+w^4q^4.
\end{align*}

Let $\eta_{n,i,j,k}$ be defined by~\eqref{full-xyz}.
Define
$$\eta_n(x,y,z,q)=\sum_{k=1}^n\sum_{i=0}^{n-k}\sum_{j=0}^{n-k}\eta_{n,i,j,k}x^i y^{j}z^{2n-2k-i-j}(xyz)^kq^k.$$
\begin{corollary}
The multivariate polynomials $\eta_n(x,y,z,q)$ are partial $e$-positive, i.e.,
$$\eta_n(x,y,z,q)=\sum_{k=1}^nq^k\sum_{2j+3\ell=0}^{2n-2k}\beta_{n,k,j,\ell}(x+y+z)^{2n-2k-2j-3\ell} (xy+xz+yz)^{j}(xyz)^{\ell+k}.$$
\end{corollary}
\section{Normal ordered grammars related to type $B$ Eulerian polynomials}\label{section04}
In the previous sections, we illustrate the basic idea of normal ordered grammars. Along the same lines as in the proof of Theorem~\ref{Eulerian-tree}, one can explore
normal ordered grammars associated with the other polynomials. In the sequel, we investigate some normal ordered grammars
related to the type $B$ Eulerian polynomials.

Let $\pm[n]=[n]\cup\{-1,-2,\ldots,-n\}$, and let $B_n$ be the hyperoctahedral group of rank $n$.
Elements of $B_n$ are signed permutations of $\pm[n]$ with the property that $\sigma(-i)=-\sigma(i)$ for all $i\in [n]$.
The {\it type $B$ Eulerian polynomials} are defined by
$$B_n(x)=\sum_{\sigma\in B_n}x^{\operatorname{des}_B(\sigma)},$$
where
$\operatorname{des}_B(\sigma)=\#\{i\in\{0,1,2,\ldots,n-1\}:~\sigma(i)>\sigma({i+1})\}$ and $\sigma(0)=0$ (see~\cite{Bre94} for details).
They satisfy the recursion (see~\cite[Eq.~(11)]{Bre94}):
\begin{equation*}\label{EulerianB}
B_{n}(x)=(1+(2n-1)x)B_{n-1}(x)+2x(1-x)\frac{\mathrm{d}}{\mathrm{d}x}B_{n-1}(x),~B_0(x)=1.
\end{equation*}
Let $B_n(x)=\sum_{k=0}^nB(n,k)x^k$. One has
\begin{equation}\label{EulerianNumberB}
B(n,k)=(1+2k)B(n-1,k)+(2n-2k+1)B(n-1,k-1),~B(0,0)=1.
\end{equation}

Let $G=\{x\rightarrow xy^2,~y\rightarrow x^2y\}$.
According to~\cite[Theorem~10]{Ma131}, we have
\begin{equation*}
D_G^n(xy)=xy^{2n+1}B_n\left(\frac{x^2}{y^2}\right),
\end{equation*}
which can be restated as
\begin{equation}\label{xDG701}
(xyD_{G'})^n(xy)=xy^{2n+1}B_n\left(\frac{x^2}{y^2}\right),~{\text{where $G'=\{x\rightarrow y,~y\rightarrow x\}$}};
\end{equation}
\begin{equation}\label{xDG702}
(xD_{G''})^n(xy)=xy^{2n+1}B_n\left(\frac{x^2}{y^2}\right),~{\text{where $G''=\{x\rightarrow y^2,~y\rightarrow xy\}$}}.
\end{equation}

It is easy to verify the following two results.
\begin{proposition}\label{Eulerian-B-tree}
If $G'=\{x\rightarrow y, y\rightarrow x\}$, then
\begin{equation*}
(xyD_{G'})^n=\sum_{k=1}^n\sum_{\ell=0}^{\lrf{(2n-k)/2}}B_{n,k,\ell}x^{k+2\ell} y^{2n-k-2\ell}D_{G'}^k,
\end{equation*}
where the coefficients $B_{n,k,\ell}$ satisfy the recurrence relation
\begin{equation}\label{B-recu}
B_{n+1,k,\ell}=(k+2\ell) B_{n,k,\ell}+(2n-k-2\ell+2)B_{n,k,\ell-1}+B_{n,k-1,\ell},
\end{equation}
with the initial conditions $B_{1,1,0}=1$ and $B_{1,k,\ell}=0$ if $(k,\ell)\neq (1,0)$.
\end{proposition}
\begin{proposition}\label{Eulerian-B-tree02}
If $G''=\{x\rightarrow y^2, y\rightarrow xy\}$, then
\begin{equation*}
(xD_{G''})^n=\sum_{k=1}^n\sum_{\ell=0}^{\lrf{(2n-k)/2}}E_{n,k,\ell}x^{k+2\ell} y^{2n-2k-2\ell}D_{G''}^k,
\end{equation*}
where the coefficients $E_{n,k,\ell}$ satisfy the recurrence relation
\begin{equation}\label{B-recu02}
E_{n+1,k,\ell}=(k+2\ell) E_{n,k,\ell}+(2n-2k-2\ell+2)E_{n,k,\ell-1}+E_{n,k-1,\ell},
\end{equation}
with the initial conditions $E_{1,1,0}=1$ and $E_{1,k,\ell}=0$ if $(k,\ell)\neq (1,0)$.
\end{proposition}

Comparing~\eqref{B-recu} with~\eqref{EulerianNumberB}, we see that $B_{n+1,1,\ell}=B(n,\ell)$,
and so we obtain a normal ordered grammatical interpretation of the type $B$ Eulerian polynomials:
\begin{equation}\label{Bnx}
B_n(x)=\sum_{\ell=0}^nB_{n+1,1,\ell}x^{\ell}.
\end{equation}

Let $$B_n(x,y,z)=\sum_{k=1}^n\sum_{\ell=0}^{\lrf{(2n-k)/2}}B_{n,k,\ell}x^{k+2\ell} y^{2n-k-2\ell}z^k.$$
It follows from~\eqref{B-recu} that
\begin{equation}\label{Bnxyz-recu}
B_{n+1}(x,y,z)=(xyz+2nx^2)B_n(x,y,z)+x(y^2-x^2)\frac{\partial}{\partial x}B_n(x,y,z),~B_0(x,y,z)=1.
\end{equation}
In particular, $$B_1(x,1,1)=x,~B_2(x,1,1)=x+x^2+x^3,~B_3(x,1,1)=x+3x^2+7x^3+3x^4+x^5.$$

We now recall two statistics of Stirling permutations.
An occurrence of an {\it ascent-plateau} of a Stirling permutation $\sigma\in\mqn$ is an index $i$ such that $\sigma_{i-1}<\sigma_{i}=\sigma_{i+1}$, where $i\in\{2,3,\ldots,2n-1\}$.
Let $\ap(\sigma)$ be the number of ascent-plateaus of $\sigma$.
The {\it flag ascent-plateau} statistic is defined by $$\fap(\sigma)=\left\{
               \begin{array}{ll}
                 2\ap(\sigma)+1, & \hbox{if $\sigma_1=\sigma_2$;} \\
                 2\ap(\sigma), & \hbox{otherwise.}
               \end{array}
             \right.
$$
Let $F_n(x)=\sum_{\sigma\in\mqn}x^{\fap(\sigma)}$.
It follows from~\cite[Eq.~(16)]{Ma20} that
\begin{equation}\label{Fnx}
F_{n+1}(x)=(x+2nx^2)F_n(x)+x(1-x^2)\frac{\mathrm{d}}{\mathrm{d}x}F_n(x).
\end{equation}
Comparing~\eqref{Fnx} with~\eqref{Bnxyz-recu}, we see that
\begin{equation}\label{Bnx11}
B_n(x,1,1)=\sum_{\sigma\in\mqn}x^{\operatorname{fap}(\sigma)}.
\end{equation}
From~\eqref{Bnx} and~\eqref{Bnx11}, we see that
Stirling permutations are closely related to signed permutations. Moreover,
by~\eqref{B-recu02}, we see that
\begin{equation}
E_{n+1,1,\ell}=(1+2\ell) E_{n,1,\ell}+(2n-2\ell)E_{n,1,\ell-1}.
\end{equation}
Using~\cite[Eq.~(6)]{Ma15}, we find that $E_{n+1,1,\ell}=\{\sigma\in\mqn: \ap(\sigma)=\ell\}$, i.e., $E_{n+1,1,\ell}$ equals the number of Stirling permutations in
$\mqn$ with $\ell$ ascent-plateaus.

Let $$E_n(x,y,z)=\sum_{k=1}^n\sum_{\ell=0}^{\lrf{(2n-k)/2}}E_{n,k,\ell}x^{k+2\ell} y^{2n-2k-2\ell}z^k.$$
It follows from~\eqref{B-recu02} that
\begin{equation*}\label{Bnxyz-recu02}
E_{n+1}(x,y,z)=(xz+2nx^2)E_n(x,y,z)+x(y^2-x^2)\frac{\partial}{\partial x}E_n(x,y,z)-x^2z\frac{\partial}{\partial z}E_n(x,y,z),
\end{equation*}
with $E_0(x,y,z)=1$.
In particular, one has
\begin{equation*}\label{Bnxyz-recu03}
E_{n+1}(1,1,z)=(z+2n)E_n(1,1,z)-z\frac{\partial}{\partial z}E_n(1,1,z).
\end{equation*}
Using~\eqref{Cn11z02} and Corollary~\ref{propT}, we arrive at
$$E_n(1,1,z)=\widetilde{C}_{n}(1,1,z)=\sum_{j=0}^{n-1}\frac{(n+j-1)!}{2^j(n-1-j)!j!}z^{n-j}~\text{for any $n\geqslant 1$},$$
and so $E_n(1,1,z)$ is the Bessel polynomial of the first kind.

Let $\pi\in\msn$. The {\it up-down runs} of a permutation $\pi\in\msn$ are the alternating runs of $\pi$ endowed with a 0
in the front. Let $\udrun(\pi)$ denote the number of up-down runs of $\pi$.
The {\it up-down run polynomials} $T_n(x)$ are defined by $T_n(x)=\sum_{\pi\in\msn}x^{\udrun(\pi)}$.
The polynomials $T_n(x)$ satisfy the recurrence relation
\begin{equation}\label{Rnx02}
T_{n+1}(x)=x(1+nx)T_{n}(x)+x\left(1-x^2\right)\frac{\mathrm{d}}{\mathrm{d}x}T_{n}(x),
\end{equation}
with initial conditions $T_0(x)=1$ and $T_1(x)=x$ ~(see~\cite{Chen2301,Ma20,Zhuang17} for details).

We end this paper by giving the following result, and omit the proof for simplicity.
\begin{proposition}\label{Eulerian-B-tree02}
Let $G'=\{x\rightarrow y, y\rightarrow x\}$.
\begin{itemize}
  \item [$(i)$] For $n\geqslant 1$, we have
\begin{equation*}
(xD_{G'})^n=\sum_{k=1}^n\sum_{\ell=0}^{\lrf{(2n-k)/2}}W_{n,k,\ell}x^{k+2\ell} y^{n-k-2\ell}D_{G'}^k,
\end{equation*}
where the coefficients $W_{n,k,\ell}$ satisfy the recurrence relation
\begin{equation*}\label{H-recu03}
W_{n+1,k,\ell}=(k+2\ell) W_{n,k,\ell}+(n-k-2\ell+2)W_{n,k,\ell-1}+W_{n,k-1,\ell},
\end{equation*}
with the initial conditions $W_{1,1,0}=1$ and $W_{1,k,\ell}=0$ if $(k,\ell)\neq (1,0)$.
  \item [$(ii)$] Let $$W_n(x,y,z)=\sum_{k=1}^n\sum_{\ell=0}^{\lrf{(2n-k)/2}}W_{n,k,\ell}x^{k+2\ell} y^{n-k-2\ell}D_{G'}^k.$$ Then we have
  \begin{equation*}
  W_{n+1}(x,y,z)=x\left(z+n\frac{x}{y}\right)W_n(x,y,z)+xy\left(1-\frac{x^2}{y^2}\right)\frac{\mathrm{d}}{\mathrm{d}x}W_n(x,y,z),
  \end{equation*}
  with the initial condition $W_0(x,y,z)=1$. In particular,
  $$W_n(1,1,z)=z(z+1)(z+2)\cdots(z+n-1)=\sum_{k=1}^n\stirling{n}{k}z^k,$$
  $$W_n(x,y,1)=y^nT_n\left(\frac{x}{y}\right),$$
  where $T_n(x)$ is the up-down run polynomial over permutations in $\msn$.
\end{itemize}

\end{proposition}

\bibliographystyle{amsplain}

\end{document}